\documentclass[11pt]{article}
\usepackage{graphicx}
\usepackage[margin=1.1in]{geometry}
\usepackage{bm}
\usepackage{subfig}
\usepackage{natbib}
\usepackage{amsmath,amssymb,amsfonts,amsthm,enumitem}

\newcommand{\N}{\ensuremath{\mathbb{N}}}
\newcommand{\R}{\ensuremath{\mathbb{R}}}

\newtheorem{lemma}{Lemma}

\newtheorem{example}{Example}

\newtheorem{condition}{Condition}
\newtheorem{definition}{Definition}

\title{Closure properties of classes of multiple testing procedures}
\author{Georg Hahn
  \\Department of Mathematics, Imperial College London}
\date{}

\begin{document}
\maketitle

\begin{abstract}
Statistical discoveries are often obtained through multiple hypothesis testing.
A variety of procedures exists to evaluate multiple hypotheses,
for instance the ones of Benjamini-Hochberg, Bonferroni, Holm or Sidak.
We are particularly interested in multiple testing procedures with two desired properties:
(solely) monotonic and well-behaved procedures.
This article investigates to which extent the classes of
(monotonic or well-behaved) multiple testing procedures,
in particular the subclasses of so-called step-up and step-down procedures,
are closed under basic set operations,
specifically the union, intersection, difference and the complement of sets of rejected or non-rejected hypotheses.
The present article proves two main results:
First, taking the union or intersection of arbitrary (monotonic or well-behaved) multiple testing procedures
results in new procedures which are monotonic but not well-behaved,
whereas the complement or difference generally preserves neither property.
Second, the two classes of (solely monotonic or well-behaved) step-up and step-down procedures
are closed under taking the union or intersection, but not the complement or difference.
\end{abstract}
\textit{Keywords:}
Multiple Hypothesis Testing, Statistical Significance, Step Up Procedure, Set Operations, Monotonicity

\section{Introduction}
\label{section_introduction}
Multiple testing is a widespread tool to evaluate scientific studies
\citep{WestfallYoung1993,Hsu1996,HochbergTamhane2008}.
We are interested in testing $m \in \N$ hypotheses $H_{01},\ldots,H_{0m}$ with corresponding p-values $p_1,\ldots,p_m$
for statistical significance
while controlling an error criterion such as the familywise error (FWER) or the false discovery rate (FDR).
Following \cite{GandyHahn2016framework}, we define a multiple testing procedure as a mapping
\begin{align*}
h: [0,1]^m \times [0,1] \rightarrow \mathcal{P}(\{ 1,\ldots,m \})
\label{label_h}
\end{align*}
whose input is a vector of $m$ p-values $p \in [0,1]^m$ and a significance level $\alpha \in [0,1]$
and whose output is the set of indices of rejected hypotheses, where $\mathcal{P}$ denotes the power set.

Many procedures of the above form are available in the literature in order to correct for multiple tests,
for instance the procedures of
\cite{Bonferroni1936}, \cite{Sidak1967}, \cite{Holm1979}, \cite{Hochberg1988} or \cite{Benjamini1995CFD}.
Many common procedures, including the ones aforementioned,
belong to a certain class of procedures, called step-up and step-down procedures \citep{RomanoShaikh2006}.
It is assumed throughout the article that only the $m$ p-values which serve as input to $h$
are used as a basis for making decisions,
dependencies between elementary hypotheses are not considered explicitly.
Apart from defining properties on $p$ imposed by those multiple testing procedures to which the results
of this article are applied, no additional conditions on $p$ are required.

This article focuses on two types of multiple testing procedures:
\textit{monotonic} procedures defined in \cite{Roth1999} and \cite{TamhaneLiu2008} as well as
\textit{well-behaved} procedures \citep{GandyHahn2016framework}.
We investigate to which extent the class of solely monotonic
and the class of well-behaved multiple testing procedures is closed
under the computation of the union, intersection, difference or the complement of sets of rejected or non-rejected hypotheses.

A multiple testing procedure is said to be monotonic if smaller p-values
\citep{TamhaneLiu2008} or a higher significance level \citep{Roth1999} lead to more rejections.
\cite{GandyHahn2016framework} call a monotonic multiple testing procedure well-behaved
if p-values corresponding to rejected hypotheses can
be lowered and p-values corresponding to non-rejected hypotheses can
be increased while leaving all rejections and non-rejections invariant.

For a set of given hypotheses,
the closed testing procedure (CTP) of \cite{Marcus1976}
(also referred to as the \textit{closure principle})
and the partitioning principle (PP) of \cite{FinnerStrassburger2002}
provide means to efficiently construct a simultaneous hypothesis test controlling the FWER.
The CTP is based on enforcing \textit{coherence} \citep{Gabriel1969}:
An intersection hypothesis $H_I$,
that is a hypothesis of the form $H_I = \cap_{i \in I} H_i$ for $I \subseteq \{1,\ldots,m\}$,
is rejected if and only if all intersection hypotheses implying $H_I$ are rejected by their local tests~\citep{Hommel2007}.
Many common procedures such as the one of \cite{Holm1979} can be constructed using the CTP.
The PP divides the parameter space underlying the hypotheses of interest into disjoint subsets
which are then tested independently at level $\alpha$.
Since the partitioned hypotheses are disjoint,
no multiplicity correction is necessary and at most one of the mutually exclusive hypotheses is true.
Whereas CTP and PP can only be used to construct procedures with FWER control,
the present article offers a means to combine procedures controlling several criteria such as the FDR
into one procedure (see the example in Section~\ref{subsection_example}).
In case of the CTP, the exponential number of tests to be carried out might also pose a problem:
The present article considers the direct construction of step-up and step-down procedures
which allow for efficient testing of multiple hypotheses.

The motivation for the present article is as follows:
\begin{enumerate}[wide]
  \item Investigating closure properties (in a set theoretical sense) of a class,
  in the case of the present article certain classes of multiple testing procedures,
is of interest in its own right:
The closure of step-up and step-down procedures
allows us to construct new multiple testing procedures of the same (step-up/step-down) form from existing ones;
moreover, the resulting procedure will be given explicitly.
  \item Being able to perform set operations with multiple testing procedures is useful in practice:
Many multiple testing procedures exist to test hypotheses according to various criteria,
each of which might prove beneficial in certain applications.
Whereas hypotheses can also be tested sequentially using several procedures,
it is nontrivial a priori that procedures can be combined to test multiple hypotheses
in a single run while drawing benefits of several criteria simultaneously.
This feature is similar to using (stepwise) ``shortcut procedures'' \citep{RomanoWolf2005,Hommel2007}
which aim to reduce the (potentially) exponential number of tests required by the CTP for FWER control
to a polynomial number of tests.
  \item Monotonic and well-behaved procedures have already been of interest in the literature.
For instance,
\cite{Gordon2007} uses the idea of monotonicity to show that there is no monotonic step-up procedure
which improves upon the \cite{Bonferroni1936} procedure in the sense that it
always returns the same rejections or possibly more.
\cite{GordonSalzman2008} show that the classical \cite{Holm1979} procedure
dominates all monotonic step-down multiple testing procedures in the above sense.
Proving that certain classes of procedures (for instance, monotonic procedures) are closed renders
the applicability of known results more apparent.
  \item The results discussed in this paper extend the methodology developed in
\cite{GandyHahn2014} and \cite{GandyHahn2016framework} which relies on well-behaved procedures.
Briefly, the authors consider a scenario in which the p-value underlying each hypothesis is unknown,
but can be estimated through Monte Carlo samples drawn under the null,
for instance using bootstrap or permutation tests.
Instead of using estimated p-values to obtain ad-hoc decisions on all hypotheses,
the authors prove that it is possible to improve
existing algorithms designed for Monte Carlo based multiple testing
\citep{BesagClifford1991,Lin2005,Wieringen2008,GuoPedadda2008,Sandve2011}:
the proposed modifications guarantee that the test results of published algorithms are identical
(up to an error probability pre-specified by the user)
to the ones obtained with the unknown p-values.
This ensures the repeatability and objectivity of multiple testing results even in the absence of p-values.
\end{enumerate}

The article is structured as follows.
Section~\ref{section_basic} provides formal definitions of the two properties of a multiple testing procedure
under investigation.
Section~\ref{section_arbitrary} considers arbitrary (solely monotonic or well-behaved) multiple testing procedures
and demonstrates that solely the monotonicity is preserved when taking unions and intersections.
The difference and complement are neither monotonic nor well-behaved.
Section~\ref{section_stepupdown} focuses on step-up and step-down procedures
and shows that both classes of (solely monotonic or well-behaved) step-up and step-down
procedures are closed under the union or intersection operation, but not the complement or difference.
The article concludes with a short discussion in Section~\ref{section_discussion}.
All proofs are given in Appendix~\ref{section_proofs}.
In the entire article, $|\cdot|$ and $\Vert \cdot \Vert$ denote the absolute value and the Euclidean norm, respectively,
and $M:=\{1,\ldots,m\}$.

\section{Basic definitions}
\label{section_basic}
Consider a step-up ($h^u$) and step-down ($h^d$) procedure
\begin{align*}
h^u(p,\alpha) &= \left\{ i \in \{ 1,\ldots,m \}: p_i \leq \max \{ p_{(j)}: p_{(j)} \leq \tau_\alpha(j) \} \right\},\\
h^d(p,\alpha) &= \left\{ i \in \{ 1,\ldots,m \}: p_i < \min \{ p_{(j)}: p_{(j)} > \tau_\alpha(j) \} \right\},
\end{align*}
returning the set of indices of rejected hypotheses \citep{GandyHahn2016framework},
where $p_{(1)} \leq p_{(2)} \leq \cdots \leq p_{(m)}$ refers to the ordered p-values.
Any procedure of the above form is fully characterised by a threshold function
$\tau_\alpha: \{ 1,\ldots,m \} \rightarrow [0,1]$ returning the critical
value $\tau_\alpha(i)$ each $p_{(i)}$ is compared to.
A step-up procedure first determines the largest $j \in M$ such that the p-value
$p_{(j)}$ lies below $\tau_\alpha(j)$ and then rejects all hypotheses having p-values up to $p_{(j)}$.
Likewise, a step-down procedure non-rejects all those hypotheses with p-values
larger or equal to the smallest p-value above the threshold function.

We now consider two useful properties of arbitrary multiple testing procedures.
The first one, monotonicity, states that smaller p-values \citep{TamhaneLiu2008}
or a higher significance level \citep{Roth1999} lead to more rejections:

\begin{definition}
\label{definition_monotonicity}
A multiple testing procedure $h$ is \textit{monotonic} if
$h(p,\alpha) \subseteq h(q,\alpha')$
for $p \geq q$ and $\alpha \leq \alpha'$.
\end{definition}
The monotonicity in $\alpha$ introduced by \cite{Roth1999},
also called $\alpha$-consistency \citep{HommelBretz2008},
is a natural property desired for any testing procedure
since testing at a more stringent significance level should never result in more rejections
\citep{DmitrienkoTamhane2013}.

\cite{GandyHahn2016framework} introduce another useful property,
the class of well-behaved multiple testing procedures.
Such procedures, in connection with a generic algorithm presented in \cite{GandyHahn2016framework},
allow to use p-value estimates obtained with independent samples under the null
to compute test results which are proven to be identical
(up to a pre-specified error probability) to the ones obtained with the unknown p-values.
A monotonic multiple testing procedure $h$ is well-behaved if it additionally satisfies the following condition.

\begin{condition}
\label{condition_h}
\begin{enumerate}
  \item Let $p,q \in [0,1]^m$ and $\alpha \in \R$.
If $q_i \leq p_i$ $\forall i \in h(p,\alpha)$ and $q_i \geq p_i$ $\forall i \notin h(p,\alpha)$,
then $h(p,\alpha) = h(q,\alpha)$.
  \item Fix $p^\ast \in [0,1]^m$ and $\alpha^\ast \in [0,1]$.
Then there exists $\delta>0$ such that $p \in [0,1]^m$, $\alpha \in [0,1]$ and
$\max( \| p-p^\ast \|, |\alpha-\alpha^\ast|) < \delta$ imply $h(p,\alpha)=h(p^\ast,\alpha^\ast)$.
\end{enumerate}
\end{condition}

Well-behaved procedures stay invariant if rejected (non-rejected) p-values are replaced by smaller (larger) values.
Moreover, well-behaved procedures are constant on a $\delta$-neighbourhood around fixed inputs $p^\ast$ and $\alpha^\ast$.

The level $\alpha$ is a parameter in Condition~\ref{condition_h} to account for settings
in which $\alpha$ is unknown a-priori:
This can occur, for instance, when the significance level depends on an estimate of the proportion of true null hypotheses
which is often a functional of $p$ \cite[Section 2.2]{GandyHahn2016framework}.
Condition~\ref{condition_h} is a generalisation of \cite[Condition 1]{GandyHahn2014}
which states the same invariance property for the case that $\alpha$ is a given constant:
In this case, $h$ is solely a function of $p$ and
the condition $|\alpha-\alpha^\ast| < \delta$ in the second part of Condition~\ref{condition_h} can be ignored.

\section{Arbitrary multiple testing procedures}
\label{section_arbitrary}
We define the union, intersection, difference and the complement of two procedures
to be the equivalent operations on the sets of rejected hypotheses returned by the two procedures.
Formally, for two multiple testing procedures $h_1$ and $h_2$ we define
\begin{align*}
&h_1 \cup h_2: [0,1]^m \times [0,1] \rightarrow \mathcal{P}(\{1,\ldots,m\}),\\
&h_1 \cup h_2(p,\alpha) := h_1(p,\alpha) \cup h_2(p,\alpha),
\end{align*}
and similarly $h_1 \cap h_2$, $h_1 \setminus h_2$ and the complement
$h_i(p,\alpha)^c := \{1,\ldots,m\} \setminus h_i(p,\alpha)$, where $i \in \{1,2\}$.

In what follows, we sometimes drop the dependence of $h(p,\alpha)$ on $p$, on $\alpha$, or on both parameters.
The following lemma summarises the results.

\begin{lemma}
\label{lemma_results_arbitrary}
Let $h_1$ and $h_2$ be two well-behaved multiple testing procedures.
\begin{enumerate}
  \item $h_1 \cup h_2$ and $h_1 \cap h_2$ are monotonic and satisfy part 2.\ of Condition~\ref{condition_h}.
  \item $h_i(p,\alpha)^c$ and $h_1 \setminus h_2$ are not monotonic, $i \in \{1,2\}$.
\end{enumerate}
\end{lemma}

As well-behaved procedures are also monotonic,
the complement or difference of two procedures is also not well-behaved.

Although by Lemma~\ref{lemma_results_arbitrary},
both the union and the intersection are monotonic,
they do not necessarily allow to lower the p-values of rejected hypotheses or
to increase the p-values of non-rejected hypotheses (first part of Condition~\ref{condition_h})
as demonstrated in the following two counterexamples.
\begin{example}
\label{counterexample_intersection}
Let $p^\ast=(0.034,0.06,1)$ and $\alpha^\ast=0.1$.
Let $h_1$ be the \cite{Benjamini1995CFD} step-up procedure,
$h_2$ be the \cite{Sidak1967} step-down procedure
and $h(p,\alpha)=h_1(p,\alpha) \cap h_2(p,\alpha)$.
Then $h_1(p^\ast,\alpha^\ast)=\{ 1,2\}$, $h_2(p^\ast,\alpha^\ast)=\{ 1\}$ and thus $2,3 \notin h(p^\ast,\alpha^\ast)$.
However, increasing $p^\ast$ to $q=(0.034,1,1)$
results in $h_1(q,\alpha^\ast)=\emptyset$ and thus $h(q,\alpha^\ast)=\emptyset \neq h(p^\ast,\alpha^\ast)$.
\end{example}

\begin{example}
\label{counterexample_union}
Let $p^\ast$ and $\alpha^\ast$ be as in Example~\ref{counterexample_intersection}.
Let $h_1$ be a step-up procedure which uses the same threshold function as the (step-down) \cite{Sidak1967} correction,
and likewise $h_2$ be a step-down procedure using the same threshold function as the (step-up) \cite{Benjamini1995CFD} procedure
-- using \citep[Lemma 3]{GandyHahn2016framework},
it is straightforward to show that both procedures are well-behaved.
Let $h(p,\alpha)=h_1(p,\alpha) \cup h_2(p,\alpha)$.
Then $h_1(p^\ast,\alpha^\ast)=\{ 1\}$, $h_2(p^\ast,\alpha^\ast)=\emptyset$ and thus $h(p^\ast,\alpha^\ast)=\{ 1\}$.
However, decreasing $p^\ast$ to $q=(0,0.06,1)$
results in $h_2(q,\alpha^\ast)=\{ 1,2 \}$ and thus $h(q,\alpha^\ast)=\{ 1,2 \} \neq h(p^\ast,\alpha^\ast)$.
\end{example}

Examples~\ref{counterexample_intersection} and~\ref{counterexample_union}
also demonstrate that both the union and the intersection of a well-behaved step-up and a well-behaved step-down procedure
are not necessarily well-behaved any more.


Although neither the class of well-behaved multiple testing procedures of general form
nor the combination of a well-behaved step-up and a well-behaved step-down procedure is closed
under the four set operations aforementioned, the next section
proves that this holds true for the special classes of well-behaved step-up and step-down procedures individually
(when taking unions and intersections).

\section{Step-up and step-down procedures}
\label{section_stepupdown}
\cite{GandyHahn2016framework} show that
any step-up or step-down procedure (characterised by its threshold function $\tau_\alpha$)
which satisfies the following condition is well-behaved:
\begin{condition}
\label{condition_threshold}
\begin{enumerate}
  \item $\tau_\alpha(i)$ is non-decreasing in $i$ for each fixed $\alpha$.
  \item $\tau_\alpha(i)$ is continuous in $\alpha$ and non-decreasing in $\alpha$ for each fixed $i$.
\end{enumerate}
\end{condition}
Furthermore, \cite{GandyHahn2016framework} verify that a large variety of commonly used procedures
satisfies Condition~\ref{condition_threshold} and is hence well-behaved,
among them the procedures of
\cite{Bonferroni1936}, \cite{Sidak1967}, \cite{Holm1979}, \cite{Hochberg1988} or \cite{Benjamini1995CFD}.

Even though \cite[Lemma 3]{GandyHahn2016framework} only prove that Condition~\ref{condition_threshold}
is sufficient for a procedure to be well-behaved, the condition is actually also necessary:

\begin{lemma}
\label{lemma_necessary}
Any well-behaved step-up or step-down procedure satisfies Condition~\ref{condition_threshold}.
\end{lemma}

Consider two step-up procedures $h^u$ and $\tilde{h}^u$
with threshold functions $\tau_\alpha^u$ and $\tilde{\tau}_\alpha^u$
as well as two step-down procedures $h^d$ and $\tilde{h}^d$
with threshold functions $\tau_\alpha^d$ and $\tilde{\tau}_\alpha^d$.

In the following subsections we separately investigate whether the classes of step-up (step-down) procedures
are closed under each of the four set operations (union, intersection, difference and complement).
Moreover, we investigate whether the subclasses of well-behaved step-up (step-down) procedures are closed.
To this end, by Lemma~\ref{lemma_necessary},
it suffices to show that the classes of step-up (step-down) procedures satisfying Condition~\ref{condition_threshold} are closed.

\subsection{Union}
\label{subsection_stepupdown_union}
\begin{figure}[!t]
  \includegraphics[width=0.49\textwidth]{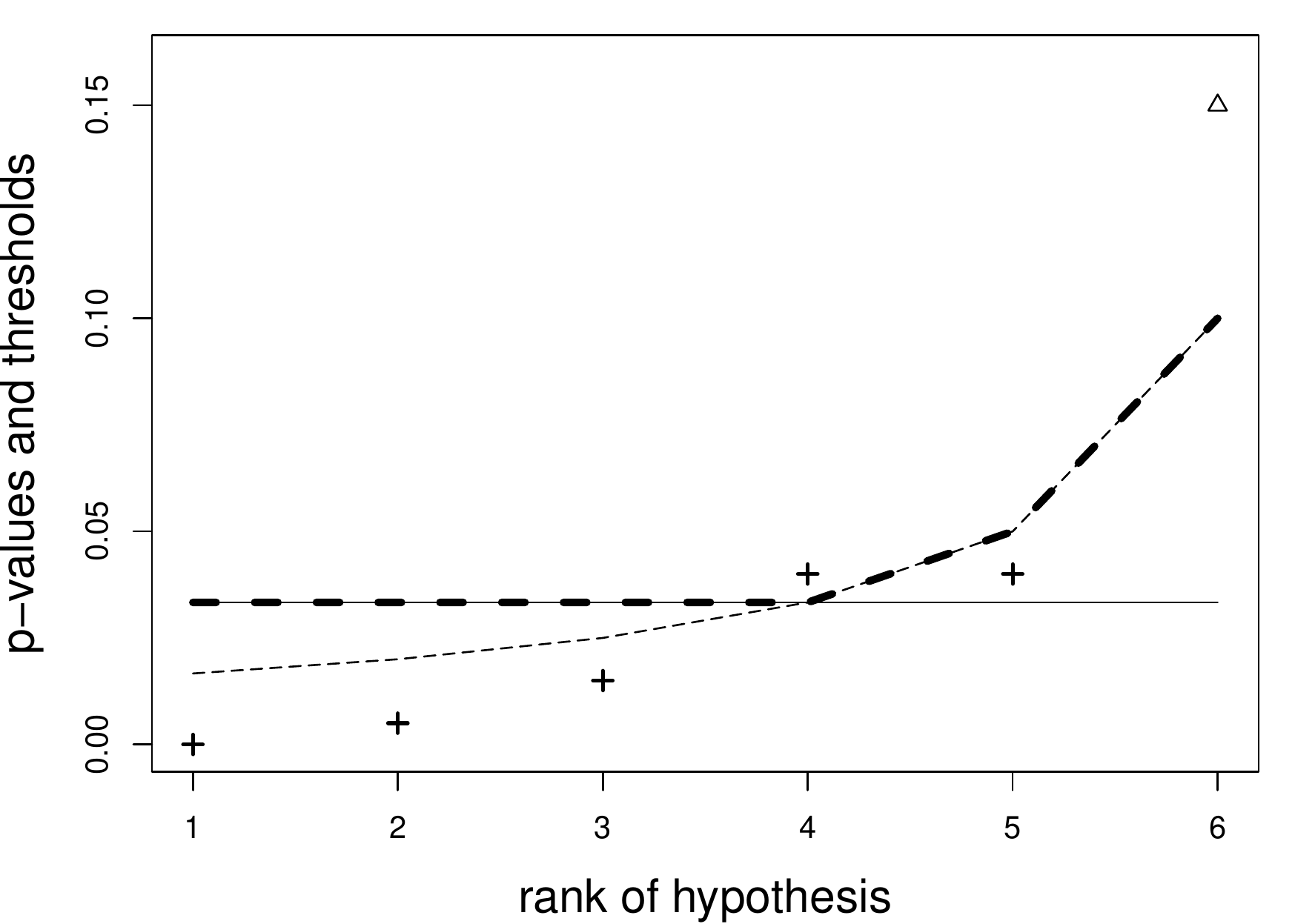}
  \includegraphics[width=0.49\textwidth]{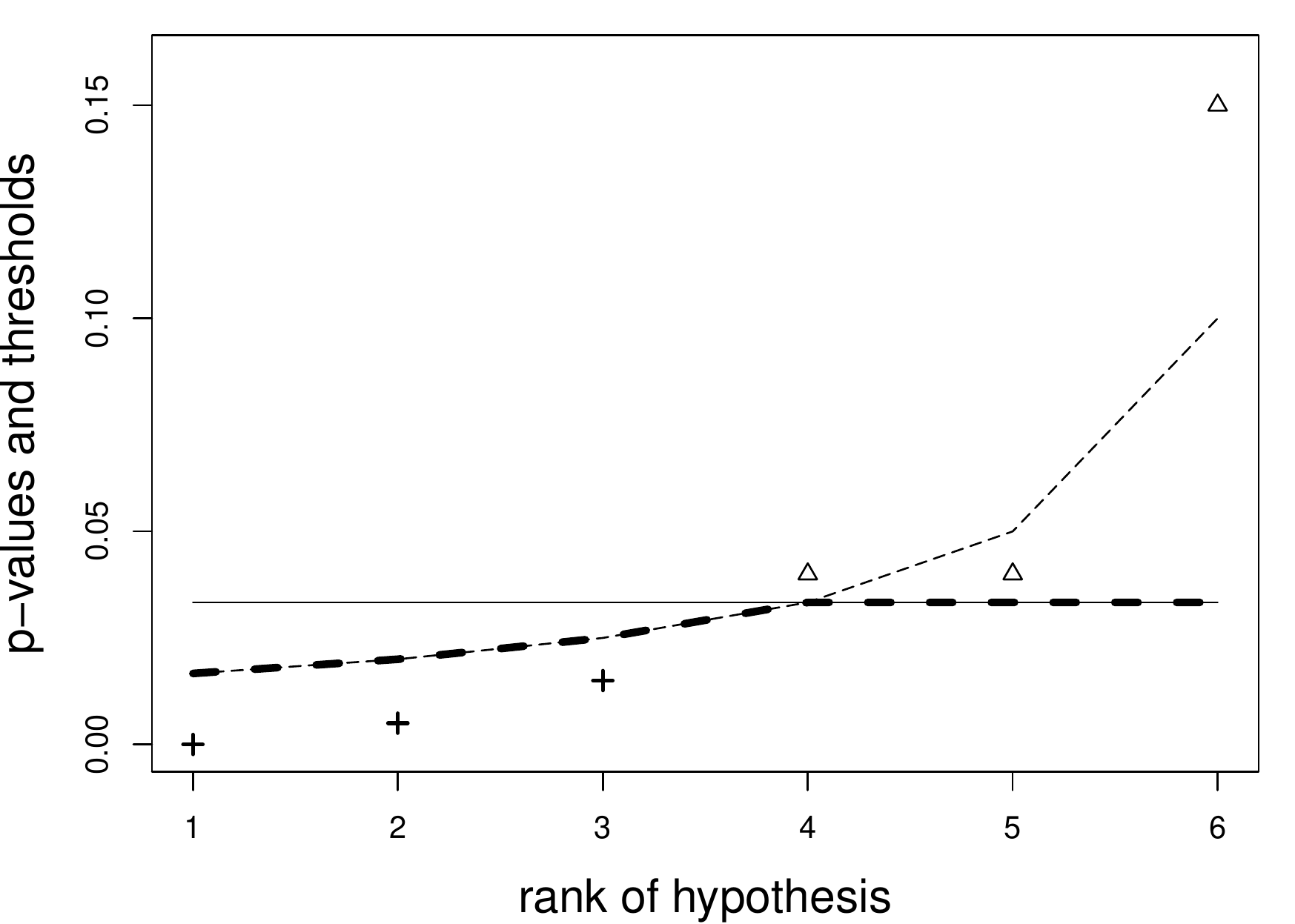}
  \caption{Combined threshold function (bold) for the computation of the
  union (left) and the intersection (right) of the \cite{Bonferroni1936} correction (vertical solid line)
  and the \cite{Hochberg1988} (dashed line) procedure.
  The \cite{Bonferroni1936} correction was applied with significance level $0.2$,
  the \cite{Hochberg1988} procedure with level $0.1$.
  P-values of rejected (crosses) and non-rejected (triangles) hypotheses.}
  \label{figure_visualisation}
\end{figure}
The class of step-up procedures is closed under the union operation:
To be precise, if $h^u$ and $\tilde{h}^u$ are two step-up procedures,
their union is computed by another step-up procedure $h$ with threshold function
$\tau_\alpha(i)=\max(\tau_\alpha^u(i),\tilde{\tau}_\alpha^u(i))$
as visualised in Fig.~\ref{figure_visualisation} (left).

This is seen as follows:
As $\tau_\alpha^u(i),\tilde{\tau}_\alpha^u(i) \leq \tau_\alpha(i)$ for all $i \in M$, all hypotheses
rejected by either $h^u$ or $\tilde{h}^u$ are also rejected by $h$,
that is $h^u \cup \tilde{h}^u \subseteq h$.
Likewise, as $\tau_\alpha(i)$ takes precisely one of the values $\tau_\alpha^u(i)$ or $\tilde{\tau}_\alpha^u(i)$ for each $i \in M$,
any p-value belonging to the non-rejection area of both procedures $h^u$ and $\tilde{h}^u$ also stays non-rejected in $h$,
hence $(h^u)^c \cap (\tilde{h}^u)^c \subseteq h^c$.

Moreover, the subclass of well-behaved step-up procedures is also closed under the union operation as proven in the following lemma.

\begin{lemma}
\label{lemma_stepup_union}
If $h^u$ and $\tilde{h}^u$ are two step-up procedures which satisfy Condition~\ref{condition_threshold}
then so does the union $h^u \cup \tilde{h}^u$.
\end{lemma}

Similarly, the union of two step-down procedures $h^d$ and $\tilde{h}^d$
(having threshold functions $\tau_\alpha^d$ and $\tilde{\tau}_\alpha^d$)
is obtained through another step-down procedure
characterised by the threshold function $\tau_\alpha(i)=\max(\tau_\alpha^d(i),\tilde{\tau}_\alpha^d(i))$.
Since the proof of Lemma~\ref{lemma_stepup_union} does not use any properties of $\tau_\alpha^u$ and $\tilde{\tau}_\alpha^u$
other than that both satisfy Condition~\ref{condition_threshold},
the maximum of two step-down threshold functions likewise leads to a threshold function satisfying Condition~\ref{condition_threshold}.

\subsection{Intersection}
\label{subsection_stepupdown_intersection}
Similarly to Section~\ref{subsection_stepupdown_union},
the intersection of two step-up procedures $h^u$ and $\tilde{h}^u$ is again a step-up procedure $h$,
characterised by the new threshold function
$\tau_\alpha(i)=\min(\tau_\alpha^u(i),\tilde{\tau}_\alpha^u(i))$
as visualised in Fig.~\ref{figure_visualisation} (right).

This is seen as follows:
As $\tau_\alpha^u(i),\tilde{\tau}_\alpha^u(i) \geq \tau_\alpha(i)$ for all $i \in M$,
any hypothesis non-rejected by either procedure $h^u$ or $\tilde{h}^u$ is also non-rejected by $h$,
that is $(h^u)^c \cup (\tilde{h}^u)^c \subseteq h^c$.
Likewise, as $\tau_\alpha(i)$ takes precisely one of the values $\tau_\alpha^u(i)$ or $\tilde{\tau}_\alpha^u(i)$ for each $i \in M$,
any p-value in the rejection area of both procedures
remains rejected when tested with $h$,
thus $h^u \cap \tilde{h}^u \subseteq h$.

Similarly to Lemma~\ref{lemma_stepup_union}, the subclass of well-behaved step-up procedures
is again closed under the intersection operation.
\begin{lemma}
\label{lemma_stepup_intersection}
If $h^u$ and $\tilde{h}^u$ are two step-up procedures which satisfy Condition~\ref{condition_threshold}
then so does the intersection $h^u \cap \tilde{h}^u$.
\end{lemma}

The intersection of two step-down procedures $h^d$ and $\tilde{h}^d$
is again obtained with another step-down procedure using the threshold function
$\tau_\alpha(i)=\min(\tau_\alpha^d(i),\tilde{\tau}_\alpha^d(i))$.
Analogously to Section~\ref{subsection_stepupdown_union},
the proof of Lemma~\ref{lemma_stepup_intersection} does not use any properties of $\tau_\alpha^u$ and $\tilde{\tau}_\alpha^u$
other than that both satisfy Condition~\ref{condition_threshold},
thus the minimum of two step-down threshold functions again leads to a threshold function satisfying Condition~\ref{condition_threshold}.

\subsection{Complement}
\label{subsection_stepupdown_complement}
Whereas the complement is generally neither well-behaved nor monotonic,
it can be computed for step-up and step-down procedures using the following construction.

Let $\alpha$ be a known constant.
We re-consider the step-up procedure $h^u$ with threshold function $\tau_\alpha^u$.
Then the step-down procedure $h^d(1-p)$ with threshold function
$\tau_\alpha^d(i)=1-\tau_\alpha^u(m+1-i)$
applied to $1-p$ (instead of $p$) computes the complement of $h^u(p)$,
where $1-p$ for $p \in [0,1]^m$ is understood coordinate-wise.

The reasoning behind this is as follows:
For any hypothesis with p-value $p_{(i)}$ below $\tau_\alpha^u(i)$,
$1-p_{(i)}$ (having rank $m+1-i$ in the sorted sequence of values $1-p$)
is above $\tau_\alpha^d(m+1-i)$ by construction of $\tau_\alpha^d$.
Therefore, all former rejections of $h^u$ turn into non-rejections of $h^d$ and vice versa.

Likewise, the complement of a step-down procedure $h^d$ with threshold function $\tau_\alpha^d$
and constant $\alpha$ is computed by a step-up procedure $h^u$ with threshold function
$\tau_\alpha^u(i)=1-\tau_\alpha^d(m+1-i)$.
Condition~\ref{condition_threshold} is again satisfied:

\begin{lemma}
\label{lemma_stepup_complement}
Let $\alpha$ be a known constant.
If the step-up procedure $h^u$ with threshold function $\tau_\alpha^u$ satisfies Condition~\ref{condition_threshold},
then so does its step-down complement $h^d$
(defined with threshold function $\tau_\alpha^d(i)=1-\tau_\alpha^u(m+1-i)$).
\end{lemma}

The requirement that $\alpha$ be a known constant is crucial since
$\tau_\alpha^d$ is not non-decreasing in $\alpha$ for a fixed $i$
as required in the second part of Condition~\ref{condition_threshold}.
However, Lemma~\ref{lemma_stepup_complement} is made possible by the fact that for a given constant $\alpha$
(that is, if $h$ and the threshold function seize to be a function of $\alpha$),
all the parts in Condition~\ref{condition_h} (and likewise, Condition~\ref{condition_threshold}) which involve $\alpha$
can be ignored (see remark at the end of Section~\ref{section_basic}).

\subsection{Difference}
\label{subsection_stepupdown_difference}
Following the notation of Section~\ref{section_arbitrary},
the difference $h_1 \setminus h_2$ of two procedures $h_1$ and $h_2$
can equivalently be written as $h_1 \cap h_2^c$ using the complement of $h_2$.
If $h_2$ is a step-up procedure, $h_2^c$ turns into a step-down procedure (see Section~\ref{subsection_stepupdown_complement}).

Therefore, in case both $h_1$ and $h_2$ are step-up (step-down) procedures satisfying Condition~\ref{condition_threshold},
Lemma~\ref{lemma_results_arbitrary} yields that
$h_1 \setminus h_2$ is still monotonic but not well-behaved any more.
However,
if $h_1$ is a step-down and $h_2$ is a step-up procedure (or vice versa),
the results from Section~\ref{subsection_stepupdown_intersection} apply and yield
that $h_1 \setminus h_2$ a well-behaved step-up/step-down procedure with explicit threshold function.

\subsection{Example}
\label{subsection_example}
Suppose we are interested in testing $H_{01},\ldots,H_{0m}$ for statistical significance while
ensuring FDR control at a pre-specified level $0.05$,
for instance using the \cite{Benjamini1995CFD} procedure.
Additionally, we are interested in only selecting those $k \in \N$ hypotheses having the lowest p-values (assuming there are no ties),
for instance due to the fact that budget constraints only allow follow-up studies for $k$ hypotheses.
We thus look to construct an intersection procedure which returns the indices of hypotheses satisfying both requirements simultaneously.

To this end,
let $h^1$ be the \cite{Benjamini1995CFD} step-up procedure controlling the FDR at level $0.05$,
defined through the threshold function $\tau^1(i) = 0.05\cdot i/m$ for $i \in \{1,\ldots,m\}$.
Moreover, let $h^2$ be the (step-up) \cite{Bonferroni1936} correction with
constant but $p$-dependant threshold function $\tau^2_p(i) = p_{(k)}$ for $i \in \{1,\ldots,m\}$,
where $p_{(k)}$ denotes the $k$'th smallest entry of vector $p=(p_1,\ldots,p_m)$.
By construction,
all rejected hypotheses by $h^2$ are precisely the ones with the $k$ lowest p-values.
Threshold functions $\tau_\alpha$ for which $\alpha=\alpha(p)$ is a function of $p$ are widely used in practice,
for instance when using an estimate of the proportion of true null hypotheses to correct the level $\alpha$
(see, for instance, Example $1$ in \cite{GandyHahn2016framework}).
Both the \cite{Benjamini1995CFD} procedure $h^1$ and the \cite{Bonferroni1936} correction $h^2$
satisfy Condition~\ref{condition_threshold} and are thus well-behaved.

Following Section~\ref{subsection_stepupdown_intersection},
the step-up procedure $h$ defined through the threshold function
$\tau_p(i)=\min(\tau^1(i),\tau^2_p(i))=\min(0.05\cdot i/m,p_{(k)})$
computes $h^1 \cap h^2$.
Moreover, $h$ is well-behaved by Lemma~\ref{lemma_stepup_intersection}.

Consider the numerical example of $15$ ordered p-values
(here denoted as $\tilde{p}$) given in Section~3.2 of \cite{Benjamini1995CFD}.
In agreement with \cite{Benjamini1995CFD},
who test $\tilde{p}$ while controlling the FDR at level $0.05$ and observe four rejections (of the first four hypotheses),
$h^1$ applied to $\tilde{p}$ yields $h^1(\tilde{p})=\{1,2,3,4\}$.
Applying the intersection procedure
$h$ constructed above with $k=3$ to $\tilde{p}$ yields $h(\tilde{p})=\{1,2,3\}$,
that is $h$ indeed yields those $k=3$ hypotheses having the lowest p-values which are also significant under FDR control at level $0.05$.

\section{Discussion}
\label{section_discussion}
This article investigates closure properties of general multiple testing procedures,
step-up and step-down procedures as well as subclasses of (solely) monotonic and well-behaved procedures
under four set operations (union, intersection, complement and difference).

The article shows that for general multiple testing procedures, solely the class of
monotonic procedures is closed under taking the union and intersection.
However, the subclass of well-behaved step-up (step-down) procedures
is closed under taking the union and intersection.

The implications of the closure properties proven in this article are threefold:
They provide a tool to construct new procedures of known form and with known properties,
they render theoretical results \citep{Gordon2007,GordonSalzman2008}
instantly applicable to a large class of multiple testing procedures
and they allow to combine the benefits of various multiple testing procedures in practice.

\appendix
\section{Proofs}
\label{section_proofs}
The appendix contains all proofs sorted by section.

\subsection{Proofs of Section~\ref{section_arbitrary}}
\begin{proof}[Proof of Lemma~\ref{lemma_results_arbitrary}]
We prove both assertions.
\begin{enumerate}[wide]
  \item Monotonicity.
If $p \leq q$ and $\alpha \leq \alpha'$ then
$h_1(q,\alpha) \subseteq h_1(p,\alpha')$, $h_2(q,\alpha) \subseteq h_2(p,\alpha')$ and thus
$h_1(q,\alpha) \cup h_2(q,\alpha) \subseteq
h_1(p,\alpha') \cup h_2(p,\alpha')$
as well as $h_1(q,\alpha) \cap h_2(q,\alpha) \subseteq h_1(p,\alpha') \cap h_2(p,\alpha')$.

The second statement of Condition~\ref{condition_h}.
As $h_1$ satisfies Condition~\ref{condition_h},
there exists $\delta_1$ such that
$\max( \| p-p^\ast \|, |\alpha-\alpha^\ast|) < \delta_1$
implies $h_1(p,\alpha)=h_1(p^\ast,\alpha^\ast)$.
Likewise for $h_2$ with a suitable $\delta_2$.
For $\delta=\min(\delta_1,\delta_2)$ and
$\max( \| p-p^\ast \|, |\alpha-\alpha^\ast|) < \delta$,
we have
$h_1(p,\alpha)=h_1(p^\ast,\alpha^\ast)$ and
$h_2(p,\alpha)=h_2(p^\ast,\alpha^\ast)$
and thus
$h_1 \cup h_2(p,\alpha)=h_1 \cup h_2(p^\ast,\alpha^\ast)$.
Likewise for the intersection.



  \item Fix $\alpha$.
If $q \leq p$ then $h_i(p,\alpha) \subseteq h_i(q,\alpha)$, but $h_i(p,\alpha)^c \supseteq h_i(q,\alpha)^c$ for $i \in \{1,2\}$.
The complement is thus not monotonic.
The operation $h_1(p,\alpha) \setminus h_2(p,\alpha)$ is equivalent to $h_1(p,\alpha) \cap (h_2(p,\alpha))^c$
and thus also not monotonic.
\end{enumerate}
\end{proof}

\subsection{Proofs of Section~\ref{section_stepupdown}}
\begin{proof}[Proof of Lemma~\ref{lemma_necessary}]
Let $h$ be a step-up (step-down) procedure characterised through its threshold function $\tau_\alpha$.
We now verify Condition~\ref{condition_threshold}.
\begin{enumerate}[wide]
\item We show that $\tau_\alpha(i)$ must be non-decreasing in $i$ for a fixed $\alpha$.
Indeed, suppose $\tau_\alpha$ is decreasing for some $i$.
Then $h$ cannot be monotonic for all inputs:
Assume that $m=2$, $p=(0.5,0.5)$ and $h$ is of step-up type with $\tau_\alpha(1)=1$ and $\tau_\alpha(2)=0$.
Then $h(p)=\{ 1\}$ but increasing $p$ to $q=(1,0.5)$ results in $h(q)=\{ 2\} \not\subseteq h(p)$,
thus contradicting monotonicity.
  \item We show that $\tau_\alpha(i)$ must also be non-decreasing in $\alpha$ for any fixed $i$.
Indeed, for a fixed $i$, suppose $\tau_\alpha(i)>\tau_{\alpha'}(i)$ for $\alpha<\alpha'$.
Then $h$ can again not be monotonic for all inputs:
Assume we test $m=1$ hypothesis $H_{01}$ with p-value $p=\tau_\alpha(1)>\tau_{\alpha'}(1)$.
Then $H_{01}$ is rejected at $\tau_\alpha(1)$ but non-rejected at $\tau_{\alpha'}(1)$ even though $\alpha<\alpha'$,
thus contradicting monotonicity.
  \item We show that $\tau_\alpha(i)$ is continuous in $\alpha$ for a fixed $i$.
Let $\epsilon>0$ be given.
Fix $i$ and $\alpha^\ast$.
We show continuity of the threshold function at $\alpha^\ast$ as $\alpha \rightarrow \alpha^\ast$.

Case 1: $\alpha^\ast>\alpha$.
Then $\tau_{\alpha^\ast}(i) \geq \tau_\alpha(i)$ by monotonicity.
Define $p^\ast=(0,\ldots,0,p_i^\ast, 1, \ldots, 1)$ for any $p_i^\ast \in [0,\tau_{\alpha^\ast}(i))$
(i.e., $p^\ast$ contains $p_i^\ast$ as $i$th entry, zeros before and ones after).
Since $h$ is well-behaved it satisfies the second part of Condition~\ref{condition_h},
hence for the fixed $p^\ast$ and $\alpha^\ast$ there exists $\delta>0$ such that for all $\alpha$ and $p$ satisfying
$|\alpha-\alpha^\ast|<\delta$, $\Vert p-p^\ast \Vert<\delta$
we have
$h(p,\alpha)=h(p^\ast,\alpha^\ast)$.
Assume $|\alpha-\alpha^\ast|<\delta$.
Define $p=(0,\ldots,0,p_i^\ast-\gamma, 1, \ldots, 1)$ for any $0<\gamma<\min(\delta,\epsilon)$.
Since $|\alpha-\alpha^\ast|<\delta$ and $\Vert p-p^\ast \Vert=\gamma<\delta$,
$h(p,\alpha)=h(p^\ast,\alpha^\ast)$ by Condition~\ref{condition_h}:
As the $i$th hypothesis is rejected in $h(p^\ast,\alpha^\ast)$ and hence also in $h(p,\alpha)$,
it follows that $\tau_{\alpha^\ast}(i) \geq \tau_\alpha(i) \geq p_i = p_i^\ast - \gamma$.
This holds true for all $p_i^\ast \in [0,\tau_{\alpha^\ast}(i))$,
thus $\tau_{\alpha^\ast}(i) \geq \tau_\alpha(i) \geq \tau_{\alpha^\ast}(i) - \gamma$
and hence $|\tau_{\alpha^\ast}(i) - \tau_\alpha(i)| \leq \gamma < \epsilon$.

Case 2: $\alpha^\ast \leq \alpha$.
Then $\tau_{\alpha^\ast}(i) \leq \tau_\alpha(i)$.
Using $p^\ast=(0,\ldots,0,p_i^\ast, 1, \ldots, 1)$ with $p_i^\ast \in (\tau_{\alpha^\ast}(i),1]$
and  $p=(0,\ldots,0,p_i^\ast+\gamma, 1, \ldots, 1)$ with $0 < \gamma < \min(\delta,\epsilon)$,
the same argument as in Case $1$ yields $\tau_{\alpha^\ast}(i) \leq \tau_\alpha(i) < \tau_{\alpha^\ast}(i) + \gamma$.
\end{enumerate}
\end{proof}

\begin{proof}[Proof of Lemma~\ref{lemma_stepup_union}]
Let $h = h^u \cup \tilde{h}^u$ be defined through $\tau_\alpha(i)=\max(\tau_\alpha^u(i),\tilde{\tau}_\alpha^u(i))$.
First, $h$ is monotonic by Lemma~\ref{lemma_results_arbitrary}.
We now verify Condition~\ref{condition_threshold}.

\begin{enumerate}[wide]
  \item The function $\tau_\alpha(i)$ is non-decreasing in $i$:
Suppose w.l.o.g.\ $\tau_\alpha(i)=\tau_\alpha^u(i)$.
If $\tau_\alpha^u(i+1) \geq \tilde{\tau}_\alpha^u(i+1)$ then
$\tau_\alpha(i) = \tau_\alpha^u(i) \leq \tau_\alpha^u(i+1) = \tau_\alpha(i+1)$
by definition of $\tau_\alpha$ as the maximum of $\tau_\alpha^u$ and $\tilde{\tau}_\alpha^u$.
If $\tau_\alpha^u(i+1) < \tilde{\tau}_\alpha^u(i+1)$ then
$\tau_\alpha(i) = \tau_\alpha^u(i) \leq \tau_\alpha^u(i+1) < \tilde{\tau}_\alpha^u(i+1) = \tau_\alpha(i+1)$.

  \item $\tau_\alpha$ is continuous in $\alpha$ as the maximum of two continuous functions
(in this case in $\alpha$) is continuous.
The function $\tau_\alpha$ is also non-decreasing in $\alpha$:
Indeed, fix $i$, let $\alpha \leq \alpha'$
and suppose w.l.o.g.\ $\tau_\alpha(i)=\tau_\alpha^u(i)$.
If $\tau_{\alpha'}^u(i) \leq \tilde{\tau}_{\alpha'}^u(i)$ then
$\tau_\alpha(i)=\tau_\alpha^u(i) \leq \tau_{\alpha'}^u(i) \leq \tilde{\tau}_{\alpha'}^u(i)=\tau_{\alpha'}(i)$
by definition of $\tau_\alpha$ as the maximum of $\tau_\alpha^u$ and $\tilde{\tau}_\alpha^u$.
Otherwise,
$\tau_\alpha(i)=\tau_\alpha^u(i) \leq \tau_{\alpha'}^u(i)=\tau_{\alpha'}(i)$.
\end{enumerate}
\end{proof}

\begin{proof}[Proof of Lemma~\ref{lemma_stepup_intersection}]
Let $h = h^u \cap \tilde{h}^u$ be defined through $\tau_\alpha(i)=\min(\tau_\alpha^u(i),\tilde{\tau}_\alpha^u(i))$.
Again, $h$ is monotonic by Lemma~\ref{lemma_results_arbitrary}.
We now verify Condition~\ref{condition_threshold}.

\begin{enumerate}[wide]
  \item The function $\tau_\alpha(i)$ is non-decreasing in $i$: Suppose w.l.o.g.\ $\tau_\alpha(i)=\tau_\alpha^u(i)$.
If $\tau_\alpha^u(i+1) \geq \tilde{\tau}_\alpha^u(i+1)$ then
$\tau_\alpha(i) = \tau_\alpha^u(i) \leq \tilde{\tau}_\alpha^u(i) \leq \tilde{\tau}_\alpha^u(i+1) = \tau_\alpha(i+1)$
by definition of $\tau_\alpha$ as the minimum of $\tau_\alpha^u$ and $\tilde{\tau}_\alpha^u$.
If $\tau_\alpha^u(i+1) < \tilde{\tau}_\alpha^u(i+1)$ then
$\tau_\alpha(i) = \tau_\alpha^u(i) \leq \tau_\alpha^u(i+1) = \tau_\alpha(i+1)$.

  \item $\tau_\alpha$ is continuous in $\alpha$ as the minimum of two continuous functions
(in this case in $\alpha$) is continuous.
The function $\tau_\alpha$ is also non-decreasing in $\alpha$:
Indeed, fix $i$, let $\alpha \leq \alpha'$
and suppose w.l.o.g.\ $\tau_\alpha(i)=\tau_\alpha^u(i)$.
If $\tau_{\alpha'}^u(i) \leq \tilde{\tau}_{\alpha'}^u(i)$ then
$\tau_\alpha(i)=\tau_\alpha^u(i) \leq \tau_{\alpha'}^u(i)=\tau_{\alpha'}(i)$.
Otherwise,
$\tau_\alpha(i)=\tau_\alpha^u(i) \leq \tilde{\tau}_\alpha^u(i) \leq \tilde{\tau}_{\alpha'}^u(i)=\tau_{\alpha'}(i)$
(by definition of $\tau_\alpha$ as the minimum).
\end{enumerate}
\end{proof}

\begin{proof}[Proof of Lemma~\ref{lemma_stepup_complement}]
Since $\tau_\alpha^u(i)$ is non-decreasing in $i$,
it is immediate to verify that $\tau_\alpha^d(i)$ is also non-decreasing in $i$.
For a given constant $\alpha$, the second part of Condition~\ref{condition_threshold} can be ignored
as shown in \citep[Condition 1]{GandyHahn2014} and is hence automatically satisfied (see Section~\ref{section_basic}).
\end{proof}


\begin{thebibliography}{}

\bibitem[Benjamini and Hochberg, 1995]{Benjamini1995CFD}
Benjamini, Y. and Hochberg, Y. (1995).
\newblock Controlling the false discovery rate: A practical and powerful
  approach to multiple testing.
\newblock {\em J Roy Stat Soc B Met}, 57(1):289--300.

\bibitem[Besag and Clifford, 1991]{BesagClifford1991}
Besag, J. and Clifford, P. (1991).
\newblock Sequential {M}onte {C}arlo p-values.
\newblock {\em Biometrika}, 78(2):301--304.

\bibitem[Bonferroni, 1936]{Bonferroni1936}
Bonferroni, C. (1936).
\newblock Teoria statistica delle classi e calcolo delle probabilit\`a.
\newblock {\em Pubblicazioni del R Istituto Superiore di Scienze Economiche e
  Commerciali di Firenze}, 8:3--62.

\bibitem[Dmitrienko and Tamhane, 2013]{DmitrienkoTamhane2013}
Dmitrienko, A. and Tamhane, A. (2013).
\newblock General theory of mixture procedures for gatekeeping.
\newblock {\em Biom J}, 55(3):402--419.

\bibitem[Finner and Strassburger, 2002]{FinnerStrassburger2002}
Finner, H. and Strassburger, K. (2002).
\newblock The partitioning principle: a powerful tool in multiple decision
  theory.
\newblock {\em Ann Stat}, 30(4):1194--1213.

\bibitem[Gabriel, 1969]{Gabriel1969}
Gabriel, K. (1969).
\newblock {Simultaneous Test Procedures -- Some Theory of Multiple
  Comparisons}.
\newblock {\em Ann Math Statist}, 40(1):224--250.

\bibitem[Gandy and Hahn, 2014]{GandyHahn2014}
Gandy, A. and Hahn, G. (2014).
\newblock {MMCTest -- A Safe Algorithm for Implementing Multiple Monte Carlo
  Tests}.
\newblock {\em Scand J Stat}, 41(4):1083--1101.

\bibitem[Gandy and Hahn, 2016]{GandyHahn2016framework}
Gandy, A. and Hahn, G. (2016).
\newblock A framework for {M}onte {C}arlo based {M}ultiple {T}esting.
\newblock {\em Scand J Stat}, 43(4):1046--1063.

\bibitem[Gordon, 2007]{Gordon2007}
Gordon, A. (2007).
\newblock Unimprovability of the {B}onferroni procedure in the class of general
  step-up multiple testing procedures.
\newblock {\em Stat Probab Lett}, 77(2):117--122.

\bibitem[Gordon and Salzman, 2008]{GordonSalzman2008}
Gordon, A. and Salzman, P. (2008).
\newblock Optimality of the {H}olm procedure among general step-down multiple
  testing procedures.
\newblock {\em Stat Probab Lett}, 78(13):1878--1884.

\bibitem[Guo and Peddada, 2008]{GuoPedadda2008}
Guo, W. and Peddada, S. (2008).
\newblock Adaptive choice of the number of bootstrap samples in large scale
  multiple testing.
\newblock {\em Stat Appl Genet Mol Biol}, 7(1):1--16.

\bibitem[Hochberg, 1988]{Hochberg1988}
Hochberg, Y. (1988).
\newblock A sharper {B}onferroni procedure for multiple tests of significance.
\newblock {\em Biometrika}, 75(4):800--802.

\bibitem[Hochberg and Tamhane, 2008]{HochbergTamhane2008}
Hochberg, Y. and Tamhane, A. (2008).
\newblock {\em Multiple Comparison Procedures}.
\newblock Wiley.

\bibitem[Holm, 1979]{Holm1979}
Holm, S. (1979).
\newblock A simple sequentially rejective multiple test procedure.
\newblock {\em Scand J Stat}, 6(2):65--70.

\bibitem[Hommel and Bretz, 2008]{HommelBretz2008}
Hommel, G. and Bretz, F. (2008).
\newblock Aesthetics and power considerations in multiple testing -- a
  contradiction?
\newblock {\em Biom J}, 50(5):657--666.

\bibitem[Hommel et~al., 2007]{Hommel2007}
Hommel, G., Bretz, F., and Maurer, W. (2007).
\newblock Powerful short-cuts for multiple testing procedures with special
  reference to gatekeeping strategies.
\newblock {\em Stat Med}, 26(22):4063--4073.

\bibitem[Hsu, 1996]{Hsu1996}
Hsu, J. (1996).
\newblock {\em Multiple Comparisons: Theory and Methods}.
\newblock Chapman and Hall/CRC.

\bibitem[Lin, 2005]{Lin2005}
Lin, D. (2005).
\newblock An efficient {M}onte {C}arlo approach to assessing statistical
  significance in genomic studies.
\newblock {\em Bioinformatics}, 21(6):781--787.

\bibitem[Marcus et~al., 1976]{Marcus1976}
Marcus, R., Peritz, E., and Gabriel, K. (1976).
\newblock On closed testing procedures with special reference to ordered
  analysis of variance.
\newblock {\em Biometrika}, 63(3):655--660.

\bibitem[Romano and Shaikh, 2006]{RomanoShaikh2006}
Romano, J. and Shaikh, A. (2006).
\newblock Stepup procedures for control of generalizations of the familywise
  error rate.
\newblock {\em Ann Stat}, 34(4):1850--1873.

\bibitem[Romano and Wolf, 2005]{RomanoWolf2005}
Romano, J. and Wolf, M. (2005).
\newblock {Exact and Approximate Stepdown Methods for Multiple Hypothesis
  Testing}.
\newblock {\em J Am Stat Assoc}, 100(469):94--108.

\bibitem[Roth, 1999]{Roth1999}
Roth, A. (1999).
\newblock {Multiple comparison procedures for discrete test statistics}.
\newblock {\em J Stat Plan Infer}, 82(1-2):101--117.

\bibitem[Sandve et~al., 2011]{Sandve2011}
Sandve, G., Ferkingstad, E., and Nygard, S. (2011).
\newblock Sequential {M}onte {C}arlo multiple testing.
\newblock {\em Bioinformatics}, 27(23):3235--3241.

\bibitem[Sidak, 1967]{Sidak1967}
Sidak, Z. (1967).
\newblock Rectangular confidence regions for the means of multivariate normal
  distributions.
\newblock {\em J Am Stat Assoc}, 62(318):626--633.

\bibitem[Tamhane and Liu, 2008]{TamhaneLiu2008}
Tamhane, A. and Liu, L. (2008).
\newblock On weighted {H}ochberg procedures.
\newblock {\em Biometrika}, 95(2):279--294.

\bibitem[van Wieringen et~al., 2008]{Wieringen2008}
van Wieringen, W., van~de Wiel, M., and van~der Vaart, A. (2008).
\newblock A test for partial differential expression.
\newblock {\em J Am Stat Assoc}, 103(483):1039--1049.

\bibitem[Westfall and Young, 1993]{WestfallYoung1993}
Westfall, P. and Young, S. (1993).
\newblock {\em Resampling-based multiple testing: Examples and methods for
  p-value adjustment}.
\newblock Wiley.

\end{thebibliography}
\end{document}